\documentclass[12pt,a4paper]{article}
\usepackage{setspace}
\usepackage{amsfonts}
\usepackage{epsfig}
\usepackage{latexsym}
\usepackage{amsthm}
\usepackage{amsmath}
\usepackage{amssymb}
\usepackage{array}
\usepackage{enumerate}
\usepackage{graphicx}
\usepackage{mathtools}
\usepackage{multirow}
\usepackage{cite}
\usepackage{ragged2e}
\usepackage{blkarray}

\DeclarePairedDelimiter\ceil{\lceil}{\rceil}
\DeclarePairedDelimiter\floor{\lfloor}{\rfloor}
\newtheorem{theorem}{Theorem}[section]

\newtheorem{lemma}{Lemma}[section]
\newtheorem{definition}{Definition}[section]
\newtheorem{remark}{Remark}[section]

\newtheorem{example}{Example}[section]
\begin{document}
\title{On the distance spectrum of minimal cages and associated distance biregular graphs}

\author{Aditi Howlader and Pratima Panigrahi
 \\ \small Department of Mathematics, Indian Institute of Technology Kharagpur, India\\ \small e-mail: aditihowlader21@gmail.com, pratima@maths.iitkgp.ernet.in}

\maketitle
\begin{abstract}
A $(k,g)$-cage is a $k$-regular simple graph of girth $g$ with minimum possible number of vertices. In this paper, $(k,g)$-cages which are Moore graphs are referred as minimal $(k,g)$-cages. A simple connected graph is called distance regular(DR) if all its vertices have the same intersection array. A bipartite graph is called distance biregular(DBR) if all the vertices of the same partite set admit the same intersection array. It is known that minimal $(k,g)$-cages are DR graphs and their subdivisions are DBR graphs. In this paper, for minimal $(k,g)$-cages we give a formula for distance spectral radius in terms of $k$ and $g$, and also determine polynomials of degree $\floor{\frac{g}{2}}$, which is the diameter of the graph. This polynomial gives all distance eigenvalues when the variable is substituted by adjacency eigenvalues. We show that a minimal $(k,g)$-cage of diameter $d$ has $d+1$ distinct distance eigenvalues, and this partially answers a problem posed in \cite{Atik1}. We prove that every DBR graph is a $2$-partitioned transmission regular graph and then give a formula for its distance spectral radius. By this formula we obtain the distance spectral radius of subdivision of minimal $(k,g)$-cages. Finally we determine the full distance spectrum of subdivision of some minimal $(k,g)$-cages.
\textbf{Subclass}: $05C12,~ 05C50$ 
\end{abstract}
\section{Introduction and Preliminaries}

In this article, by a graph we mean a finite, simple, connected and undirected graph. Let $G=(V,E)$ be a graph with vertex set $V(G)=\{v_{1},v_{2},\ldots,v_{n}\}$ and edge set $E(G)=\{e_{1}, e_{2},\ldots, e_{m}\}$. The \emph{adjacency matrix} $A(G)$ of $G$ is an ${n\times n}$ matrix with $(i,j)^{th}$ entry $1$ or $0$ according as $v_{i}$ is adjacent to $v_{j}$ or not. 
The \emph{incidence matrix} $R(G)$ of $G$ is an ${n\times m}$ matrix whose  $(i,j)^{th}$ entry is $1$ or $0$ according as vertex $v_{i}$ is an end vertex of edge $e_{j}$ or not.
The \emph{distance matrix} $D(G)$  of $G$ is an ${n \times n}$ matrix whose  $(i,j)^{th}$ entry is the distance (length of the shortest path) between the vertices $v_{i}$ and $v_{j}$.
The eigenvalues of $A(G)$ (respectively $D(G)$) are called eigenvalues (respectively distance eigenvalues or D-eigenvalues) of $G$. 
The set of all eigenvalues (respectively distance eigenvalues) of $G$ is called the spectrum (respectively distance spectrum or D-spectrum) of $G$.
If $\lambda_{1}, \lambda_{2},\ldots,\lambda_{p}$ are distinct eigenvalues (respectively distance eigenvalues) of $G$ with respective multiplicities $m_1, m_2,\ldots,m_p$ then the spectrum (respectively distance spectrum) of $G$ is denoted by $\{\lambda_{1}^{(m_1)}, \lambda_{2}^{(m_2)},\ldots,\lambda_{p}^{(m_p)}\}$. 
The largest eigenvalue of $D(G)$ is called the \emph{distance spectral radius} of $G$. \\

\vspace{-1em}
 The distance matrix of a graph gives several structural information of the graph. Thus the computation of the distance matrix and its characteristic polynomial is much more intense problem. Graham and Pollack \cite{Gra} introduced distance matrix of a graph and established a relationship between the number of negative eigenvalues of this matrix and addressed a problem in data communication systems. The distance matrix and distance spectrum of a graph has numerous applications to chemistry \cite{Bala} and other branches of science and engineering. For some recent results on the characteristic polynomials of the distance matrices and distance spectra of graphs, one may refer \cite{Aou,Ali,Lin}. \\

\vspace{-1em}
 For any graph $G$ of diameter $d$, and a vertex $u\in V(G)$, $G_{i}(u)$ denotes the set of all vertices in $G$ of distance $i$ from  $u$, $i = 0, 1,\ldots, d$. A connected graph $G$ is called \emph{distance regular} (in short DR) if it is regular and for any two vertices $x, y \in G$ at distance $i$, there are precisely $c_{i}$ neighbors of $y$ in $G_{i-1}(x)$ and $b_{i}$ neighbors of $y$ in $G_{i+1}(x)$, $0\leq i\leq d$, where $c_{0}$ and $b_{d}$ are undefined. The sequence $\{b_{0}, b_{1},\ldots, b_{d-1}; c_{1}, c_{2},\ldots, c_{d}\}$ is called the \emph{intersection array} of a distance regular graph $G$. 
 For $i = 0, 1,\ldots, d$, the numbers $c_{i}$, $b_{i}$, and $a_{i}$, where $a_{i} = l-b_{i}-c_{i}$ and $l$ is the degree of regularity of $G$, are called the \emph{intersection numbers} of $G$. Biggs \cite{Big} introduced distance regular (DR) graphs. 
 For results on DR graphs and their link with other combinatorial structures one may refer \cite{BroDR,van}. Every DR-graph of diameter $d$ has exactly $d+1$ distinct adjacency eigenvalues and at most $d+1$ distinct D-eigenvalue \cite{Atik1}. The authors in \cite{Ala} characterized some DR graphs with diameter three and four having exactly three distinct distance eigenvalues.\\
 
\vspace{-1em}
 For an $n$-vertex graph $G$ with diameter $d$, the \emph{$i^{th}$ distance matrix} $A_{i}$, $i=1,2,\ldots,d$, of $G$ is an $n\times n$ matrix whose rows and columns are indexed by vertices of $G$ and $(j,m)^{th}$ entry is $1$ or $0$ according as distance between $j^{th}$ and $m^{th}$ vertices is $i$ or not. Thus the distance matrix $D$ of graph $G$ can be written as
 \begin{equation}
 D=A_{1}+2A_{2}+\cdots+d A_{d}
 \label{eq1}
\end{equation}
The adjacency matrix $A$ of a distance regular graph $G$ with diameter $d$ and its $i^{th}$ distance matrices satisfy the following recurrence relation \cite{BroDR}.
\begin{equation}
AA_{i}=c_{i+1}A_{i+1}+a_{i}A_{i}+b_{i-1} A_{i-1},~~A_{0}=I,~~A_{1}=A, ~~i=0,1,2,\ldots,d
\label{eq2}
\end{equation}
Applying equation (\ref{eq2}) we get that the $i^{th}$ distance matrix $A_{i}$ of a distance regular graph $G$ with diameter $d$ can be expressed as a polynomial (of degree $i$) of its adjacency matrix $A$, $i=1,2,\ldots,d$. Then from (\ref{eq1}), the distance matrix $D$  can also be written as a polynomial of $A$, say $D=p(A)$, of degree $d$. Thus for every eigenvalue $\lambda$ of $A$, $p(\lambda)$ is a distance eigenvalue of graph $G$.\\

 A connected graph $G$ is called \emph{distance-biregular (DBR) graph} if it is bipartite and all vertices in the same partite set have the same intersection array. We denote the bi-partition of a DBR graph as $(V_{1},V_{2})$. The intersection arrays for vertices in $V_1$ and $V_2$ are  $\{r, e_{1},\ldots, e_{d_{1}-1}; 1, f_{2},\ldots, f_{d_{1}}\}$ and  $\{s, g_{1},\ldots, g_{d_{2}-1}; 1, h_{2},\ldots, h_{d_{2}}\}$ respectively, where $r$ is the degree of vertices in $V_1$, $s$ is the degree of vertices in $V_{2}$, $d_{1}=$max $\{d(x,y):x\in V_{1},y\in V(G)\}$ and $d_{2}=$max $\{d(x,y):x\in V_{2},y\in V(G)\}$. The diameter $d^{\prime}$ of $G$ is of course $max(d_{1}, d_{2})$. For any $u \in V_1$ and $v \in V_2$ we take $l_{i}=|G_{i}(u)|$ and  $l_{i}^{\prime}=|G_{i}(v)|$, $i=0,\ldots,d^{\prime}$. We note that $l_{d^{\prime}-1}\neq 0$ and $l_{d^{\prime}-1}^{\prime}\neq 0$ though one of $l_{d^{\prime}}$ and $l_{d^{\prime}}^{\prime}$ may be zero.

Some elementary relations on the intersection arrays of a DBR graph are given below.
\lemma \rm (\cite{God}) \label{ldbr}
For a DBR graph, the following relations hold true:
    $l_{0}=1$, $l_{i+1}f_{i+1}=l_{i}e_{i}$, $l_{0}^{\prime}=1$, and $l_{i+1}^{\prime} h_{i+1}=l_{i}^{\prime} g_{i}$.\\

For any graph $G$ and a vertex $v$ in it, the \emph{transmission} $Tr_{G}(v)$ of $v$ is the sum of distances from $v$ to all other vertices in $G$. A connected graph $G$ is called \emph{$p$-transmission regular} if $Tr_{G}(v)$ is $p$ for all the vertices $v$ in $G$.
It is known \cite{BroDR} that for any vertex $u$ in a DR graph $G$, $G_{i}(u)$ has a constant number of vertices, say $k_{i}$, $i =0,1,\ldots,d$. Also $k_{i}$ satisfies the relations
$k_{0} = 1,~ k_{1} = l,~ k_{i+1} c_{i+1} = k_{i}b_{i}$ for $i=0,1,\ldots, d-1$. Thus any DR graph is a $p$-transmission regular graph, where $p =\sum_{i=0}^{d} i k_{i}$.  We note that the distance spectral radius of every $p$-transmission regular graph is equal to $p$.

\definition \rm \cite{Bro}
Suppose $A$ is a real symmetric matrix whose rows and columns are indexed by elements in $X=\{1,2,\ldots,n\}$. Consider the block representation of $A$ with respect to the partition $\{X_{1},X_{2},\ldots,X_{m}\}$ of $X$ as $A$=
 $\begin{pmatrix}
  A_{11} & A_{12} &\cdots& A_{1m} \\
  A_{21} & A_{22} &\cdots& A_{2m} \\
  \cdots & \cdots &\cdots& \cdots  \\
  A_{m1} & A_{m2} &\cdots& A_{mm}
  \end{pmatrix} $,
  where each $A_{ij}$ denotes the sub-matrix (block) of $A$ formed by rows indexed in $X_{i}$ and the columns indexed in $X_{j}$. Let $q_{ij}$ be the average row sum of $A_{ij}$. Then the matrix $Q=(q_{ij})$ is called a \emph{quotient matrix} of $A$. For each block $A_{ij}$, if the row sum is constant then the partition is called \emph{equitable}.
 \lemma \rm (\cite{Bro}) \label{q1}
 Let $Q$ be a quotient matrix of a real symmetric matrix $A$ corresponding to an equitable partition. Then the spectrum of $A$ contains the spectrum of $Q$.
 
 \lemma \rm (\cite{Atik}) \label{q2}
 If $Q$ is a quotient matrix of a real symmetric matrix $A$ corresponding to an equitable partition, then the largest eigenvalue of $A$ is equal to the largest eigenvalue of $Q$.
 
A connected graph $G$ is called a \emph{$t$-partitioned transmission regular graph} if there exists a partition $\bigcup_{i=1}^{t} U_{i}$ (called a $t$-partition) of the vertex set of $G$ such that for any $i,j$ (not necessarily distinct) in  $\{1,2,\ldots,t\}$ and for any vertex $x\in U_{i}$, $q_{ij}=\sum_{y\in U_{j}} d(x,y)$ is a constant, where $d(x,y)$ is the distance between $x$ and $y$ in the graph $G$. For a $t$-partitioned transmission regular graph $G$, $\{U_{i}: i= 1,2,\dots,t\}$ is an equitable partition of $D(G)$.
Therefore the quotient matrix of $D(G)$ with respect to this partition is $Q^{D}=[q_{ij}]_{t\times t}$, and so by Lemma \ref{q2} the distance spectral radius of $G$ is the largest eigenvalue of $Q^{D}$.\\

\vspace{-1em}
 For positive integers $k$ and $g$, a $(k,g)$-cage  is a $k$-regular simple graph of girth $g$ on minimum possible number, say $n(k,g)$, of vertices. An well known \cite{Big} lower bound for $n(k,g)$ is as given below.
{\small{\begin{align*}
n(k,g)
&\geq n_{0}(k,g)=
\begin{cases}
1+k+k(k-1)+\cdots+k(k-1)^{d-2}+k(k-1)^{d-1}, &\text{if $g$ is odd}\\
   1+k+k(k-1)+\cdots+k(k-1)^{d-2}+(k-1)^{d-1}, &\text{if $g$ is even},
   \end{cases}
\end{align*}}}
where $d=\floor{\frac{g}{2}}$ is the diameter of the $(k,g)$-cage. A $(k,g)$-cage for which equality holds in the above bound is called a \emph{Moore graph} or a \emph{minimal $(k,g)$-cage.}\\

\vspace{-1em}
The Lemma below gives information about the possible minimal $(k,g)$-cages.
\lemma \rm (\cite{Exo}) \label{lexist}
There exists a Moore graph (or a minimal $(k,g)$-cage) of degree $k$ and girth $g$ if and only if\\
$(i)$ $k= 2$ and $g>3$, cycles;\\
$(ii)$ $g= 3$ and $k>2$, complete graphs;\\
$(iii)$ $g= 4$ and $k>2$, complete bipartite graphs;\\
$(iv)$ $g= 5$ and:\\
$~~~~~~~~~k=2$, the $5$-cycle,\\
$~~~~~~~~~k=3$, the Petersen graph,\\
$~~~~~~~~~k=7$, the Hoffman-Singleton graph,and possibly $k= 57$;\\
      $(v)$ $g= 6,8$, or $12$, and there exists a symmetric generalized $n$-gon of order $k-1$.
\\

\vspace{-1em}
It is known \cite{Big} that every minimal $(k,g)$-cage is a DR graph with intersection array
$\{k ,k-1,\ldots,k-1,k-1;1 ,1 ,\ldots,1, k\}$ if $g$ is even, and
$\{k ,k-1,\ldots,k-1,k-1;1 ,1 ,\ldots,1, 1\}$ if $g$ is odd. So the intersection number $a_i=0$ for all minimal $(k,g)$-cages, $i=0,1,\ldots,d$. The \emph{subdivision graph $S(G)$} of the graph $G$ is obtained from $G$ by inserting a new vertex of degree $2$ on each edge of $G$.\\

\vspace{-1em}
The result below gives adjacency spectrum of minimal $(k,g)$-cages. 
\lemma \rm (\cite{Big}) \label{lspec}
Let $G$ be a $(k,g)$-cage with diameter $d$ and $n$ vertices.\\
$(i)$ If $g = 2d$ then the $d+1$ distinct eigenvalues of $G$ are
$\lambda =  k,-k,2 \sqrt{k-1} \cos \frac{\pi j}{d}$, $j =1,2,\ldots, d-1$, with multiplicity  $m_{\lambda}=\frac{nk}{g}[\frac{4h-\lambda^{2}}{k^{2}-\lambda^{2}}]$, $h=k-1$, $|\lambda|\neq k$.\\
$(ii)$ If $g = 2d + 1$ then the $d+1$ distinct eigenvalues of $G$ are
$\lambda = k, 2 \sqrt{k-1} \cos a_{j}$, $j = 1,2,\ldots,d$,
where  $a_{1},\ldots,a_{d}$ are the distinct solutions in the interval $0<a<\pi$
of the equation $\sqrt{k-1} \sin (d+1)a + \sin d a = 0$ with multiplicity of an eigenvalue $\lambda$ is given by $m_{\lambda} =\frac{nk}{g}[\frac{4h-\lambda^{2}}{(k-\lambda)(f+\lambda)}]$, $h=k-1,~f=k+\frac{k-2}{g}$.\\

\vspace{-1em}
 In this paper, for minimal $(k,g)$-cages we give a formula for distance spectral radius in terms of $k$ and $g$, and also determine polynomials of degree $\floor{\frac{g}{2}}$ which give all distance eigenvalues when the variable is substituted by adjacency eigenvalues. The authors in \cite{Atik1} proved that every DR graph with diameter $d$ has at the most $d + 1$ distinct D-eigenvalues and then asked for  characterization of DR graphs which will have exactly $d + 1$ distinct D-eigenvalues. We show that all minimal $(k,g)$-cages of diameter $d$ have $d+1$ distinct distance eigenvalues. In \cite{God} it is proved that every distance-regularized graph is either DR or DBR. 
The authors in \cite{Mohar} proved that subdivision of a minimal $(k,g)$-cage is a DBR graph. We prove that every DBR-graph is a $2$-partitioned transmission regular graph and then give a formula for its distance spectral radius. By this formula we determine distance spectral radius of subdivision of minimal $(k,g)$-cages. We also find D-spectrum of subdivision of minimal $(3,5)$-cages, minimal $(3,6)$-cages, and minimal $(k,g)$-cages for $g=3$ and $4$ with any values of $k \geq 3$.\\
 
\vspace{-1em}
 Next we state some known results which will be used in the sequel.
 
\definition \rm \cite{Sc}
Let $A=(a_{ij})$ be an $m\times n$ matrix and $B=(b_{ij})$ be a $p\times q$ matrix then the \emph{Kronecker product} of $A$ and $B$, denoted by $A\otimes B$, is defined as the $mp\times nq$ partition matrix $(a_{ij}~B)$. The product of two kronecker products gives another kronecker product: $(M\otimes P)(N\otimes Q)=MN\otimes PQ$, in case where each multiplication makes sense.\\
Recall that for any graph $G$, its line graph $L(G)$ is the graph whose vertex set is $E(G)$ and two vertices are adjacent if the corresponding edges in $G$ share a common end vertex.

 \lemma \rm (\cite{Cve}) \label{l1}
Let $G$ be an $r$-regular graph with adjacency matrix $A$, incidence matrix $R$, and line graph $L(G)$. Then $RR^{T} = A+rI$, $R^{T}R = A(L(G))+2I$, $JR = 2J = R^{T} J$ and $JR^{T} = r J = R J$, where $I$ is the identity matrix and $J$ is the all-one matrix of appropriate order.

\lemma \rm (\cite{Cve}) \label{l2}
Let G be an $r$-regular graph with $p$ vertices, $q$  edges, and eigenvalues $\{r,~ \lambda_{2},~\ldots,~\lambda_{p}\}$. Then spectrum of $L(G)$ is $\{2r-2,~ \lambda_{2}+r-2,~ \ldots,~ \lambda_{p}+r-2,~ -2^{(q-p)}\}$.
Also, $Z$ is an eigenvector corresponding to the eigenvalue $-2$ if and only if $RZ = 0$, where $R$ is the incidence matrix of $G$.

\section{\textbf{Distance spectrum of minimal cages}}
Here we give a formula for distance spectral radius of minimal $(k,g)$-cages.
\begin{theorem}
The distance spectral radius of a minimal $(k,g)$-cage, $k \geq 3$, is\\
$~~~~~~~~~~~~~~~~~~\lambda_{1}=\begin{cases}
\frac{k\{1-(k-1)^{d}\}}{(2-k)^{2}}-\frac{2d(k-1)^{d}}{(2-k)},& \text{$g$ even}\\
\frac{k\{1-(k-1)^{d}\}}{(2-k)^{2}}-\frac{d k(k-1)^{d}}{(2-k)},& \text{$g$ odd}
\end{cases}
$, where $d=\floor{\frac{g}{2}}$.
\end{theorem}
\begin{proof}
Since a minimal $(k,g)$-cage is a DR graph which is also a $p$-transmission regular graph, the distance spectral radius of this graph is the transmission $p$ of any vertex $x$ in it. From the intersection array of the minimal $(k,g)$-cage we get, $p=\sum_{y\in G} d(x,y)=  k+2k(k-1)+3k(k-1)^{2}+\cdots+(d-1)k(k-1)^{d-2}+dc(k-1)^{d-1}$, where $c=1$ for $g$ even and $c=k$ for $g$ odd. For $g$ even, $p=k+2k(k-1)+3k(k-1)^{2}+\cdots+(d-1)k(k-1)^{d-2}+d(k-1)^{d-1}=k+2k(k-1)+3k(k-1)^{2}+\cdots+(d-1)k(k-1)^{d-2}+d(k-k+1)(k-1)^{d-1}= k[1+2(k-1)+\cdots+d(k-1)^{d-1}]-d(k-1)^{d}=kS-d(k-1)^{d}$, where $S=1+2(k-1)+\cdots+d(k-1)^{d-1}.$ Then we get $S-(k-1)S=[1+(k-1)+\cdots+(k-1)^{d-1}]-d(k-1)^{d}$. So $S=\frac{1-(k-1)^{d}}{(2-k)^{2}}-\frac{d(k-1)^{d}}{(2-k)}$ and we get the result in this case. If $g$ is odd then $p=k S$, and hence the result.
\end{proof}
The next lemma will be useful to prove some important results of this paper.
\begin{lemma} \label{lemmainteger}
For integers $i$ and $j$, $i,j= 0,1,2,\ldots,d$, consider the recurrence relation
$a_{i}^{j}=
\begin{cases}
a_{i-1}^{j}+a_{i-2}^{j-1}, & \text{if $1\leq j\leq \floor{\frac{i}{2}}$}\\
0, & \text{otherwise}
\end{cases}$,  with initial conditions\\ $a_{i}^{0}=\begin{cases}
k-1, & \text{i=1,\ldots,d}\\
k, & \text{i=0}
\end{cases}$.
Then we get,
\begin{enumerate} [(i)]
    \item for $j>0$, $a_{2j+b}^{j}$=$\begin{cases}
   k, & \text{if $b=0$}\\
   k+a_{2j-1}^{j-1}+a_{2j}^{j-1}+\cdots+a_{2j+b-2}^{j-1}, & \text{if $b>0$},
   \end{cases}$\\
   \item $a_{i}^{j}=
   g_{i}^{j} k-h_{i}^{j}$ for $1\leq j< \floor{\frac{i}{2}}$, $i=4,\ldots,d$,
    where $g_{i}^{j}=1+g_{2j-1}^{j-1}+g_{2j}^{j-1}+\cdots+g_{i-2}^{j-1}$, $h_{i}^{j}=h_{2j-1}^{j-1}+h_{2j}^{j-1}+\cdots+h_{i-2}^{j-1}$, $g_{i}^{1}=i-1$, $h_{i}^{1}=i-2$, and $g_{i}^{0}=h_{i}^{0}=1$.
\end{enumerate}
\end{lemma}
\begin{proof}
$(i)$ First, we take $b=0$, do induction on $j$ and show that $a_{2j}^{j}=k$. For $j=0$, $a_0^0=k$, and for $j=1$, we have $a_{2}^{1}=a_{1}^{1}+a_{0}^{0}=k$, as $a_{1}^{1}=0$ from the hypothesis. We assume that the result is true up to $j-1$. Now $a_{2j}^j=a_{2j-1}^{j}+a_{2j-2}^{j-1}=0+k$, since $a_{2j-1}^{j}=0$ from the hypothesis. Hence  $a_{2j}^{j}=k$, for every $j$. Next let $b>0$. For any fixed $j\geq 1$ we do induction on $b$. If $b=1$, $a_{2j+1}^{j}=a_{2j}^{j}+a_{2j-1}^{j-1}=k+a_{2j-1}^{j-1}$. We assume that the equation holds true up to  $b-1$. Now $a_{2j+b}^{j}= a_{2j+b-1}^{j}+a_{2j+b-2}^{j-1}=k+a_{2j-1}^{j-1}+a_{2j}^{j-1}+\cdots+a_{2j+b-3}^{j-1}+a_{2j+b-2}^{j-1}$. This proves the first part of the lemma.
\vspace{2mm} 
 \par For $(ii)$, we do induction on $j$ for any fixed $i\geq 2$. By $(i)$ of this Lemma we get,
$a_{i}^{1}=k+a_{1}^{0}+a_{2}^{0}+\cdots+a_{i-2}^{0}=k+(i-2)(k-1)=(i-1)k-(i-2)=g_{i}^{1}k-h_{i}^{1}$. This proves the result for $j=1$. Let the equation be true upto $j-1$. Now $a_{i}^{j}=k+a_{2j-1}^{j-1}+a_{2j}^{j-1}+\cdots+a_{i-2}^{j-1}=k+(g_{2j-1}^{j-1}k-h_{2j-1}^{j-1})+(g_{2j}^{j-1}k-h_{2j}^{j-1})+\cdots+(g_{i-2}^{j-1}k-h_{i-2}^{j-1})=(1+g_{2j-1}^{j-1}+g_{2j}^{j-1}+\cdots+g_{i-2}^{j-1})k-(h_{2j-1}^{j-1}+h_{2j}^{j-1}+\cdots+h_{i-2}^{j-1})=g_{i}^{j}k-h_{i}^{j}$. Hence the result.
\end{proof}

\begin{theorem} \label{theoremai}
 Let $G$ be a minimal $(k,g)$-cage. The $i^{th}$ distance matrix $A_{i}$, $i=0,1,2,\ldots,d,$ of $G$ can be expressed as:
\begin{align}
A_{i}=&\frac{1}{c}[A^{i}- a_{i}^{1} A^{i-2}+ (k-1)a_{i}^{2} A^{i-4}-(k-1)^{2} a_{i}^{3} A^{i-6}+\cdots+(-1)^{\floor{\frac{i}{2}}} \nonumber\\
&(k-1)^{{\floor{\frac{i}{2}}}-1}a_{i}^{{\floor{\frac{i}{2}}}}A^{i-2{\floor{\frac{i}{2}}}}]
\label{eq3}
\end{align}
where $a_{i}^{j}$ are as in Lemma \ref{lemmainteger}, $c$ is  $k$ for $i=d$ and $g$ even, and is $1$ otherwise.
\end{theorem}
\begin{proof}
We do induction on $i$. First, let $g$ be an odd integer. So intersection array of $G$ is $\{k,k-1,\ldots,k-1;1, 1,\ldots,1\}$. From recurrence relation (\ref{eq2}), we have $ A A_{1}=c_{2}A_{2}+a_{1}A_{1}+b_{0}A_{0}$, $A_{0}=I$ and $A_{1}=A$.
  Since $a_{1}=0,~c_{2}=1,$ and $b_{0}=k$, we get $A^2=A_{2}+k I$. Then $A_{2}= A^{2}-k I=A^{2}-a_{2}^{1}I$. Thus equation (\ref{eq3}) is true for $i=0,1,2$.
 Let us assume that it is true up to $d-1$. Then we consider $i=d$. 
Since $a_{d-1}=0,~c_{d}=1,$ and $b_{d-2}=k-1$, we have
\begin{align*}
 A A_{d-1}=c_{d}A_{d}+a_{d-1}A_{d-1}+b_{d-2}A_{d-2} = A_{d}+(k-1) A_{d-2},
 \end{align*}
 \begin{align*}
 &A( A^{d-1}- a_{d-1}^{1} A^{d-3}+ \cdots+(-1)^{\floor{\frac{d-1}{2}}}(k-1)^{{\floor{\frac{d-1}{2}}}-1} a_{d}^{{\floor{\frac{d-1}{2}}}} A^{d-1-{\floor{\frac{d-1}{2}}}})= A_{d}+\\
&(k-1)(A^{d-2}-a_{d-2}^{1} A^{d-4}
+\cdots+(-1)^{\floor{\frac{d-2}{2}}}(k-1)^{{\floor{\frac{d-2}{2}}}-1} a_{d-2}^{{\floor{\frac{d-2}{2}}}}
A^{d-2-2{\floor{\frac{d-2}{2}}}}).
\end{align*}
 Since for $d$ even, $\floor{\frac{d-2}{2}}= \floor{\frac{d-1}{2}}=\floor{\frac{d}{2}}-1$, we get\\
\begin{equation*}
\begin{split}
A_{d}=&A^{d}-\{a_{d-1}^{1}+(k-1)\} A^{d-2}+(k-1)\{a_{d-1}^{2}+a_{d-2}^{1}\} A^{d-4}-\cdots+
(-1)^{\floor{\frac{d}{2}}}\\
&(k-1)^{{\floor{\frac{d}{2}}-1}}
a_{d-2}^{\floor{\frac{d-2}{2}}}A^{d-2{\floor{\frac{d}{2}}}}\\
=&A^{d}-a_{d}^{1} A^{d-2}+(k-1)a_{d}^{2} A^{d-4}-\cdots+(-1)^{\floor{\frac{d}{2}}}(k-1)^{{\floor{\frac{d}{2}}-1}}a_{d}^{\floor{\frac{d}{2}}}A^{d-2{\floor{\frac{d}{2}}}}.\\
\end{split}
\end{equation*}
For $d$ odd, $\floor{\frac{d-2}{2}}=\floor{\frac{d}{2}}-1$ and  $\floor{\frac{d-1}{2}}=\floor{\frac{d}{2}}$. So we get
\begin{equation*}
\begin{split}
A_{d}= &A^{d}-\{a_{d-1}^{1}+(k-1)\} A^{d-2}+(k-1)\{a_{d-1}^{2}+a_{d-2}^{1}\} A^{d-4}-\cdots+(-1)^{\floor{\frac{d}{2}}}\\
&(k-1)^{{\floor{\frac{d}{2}}-1}}\{a_{d-1}^{\floor{\frac{d}{2}}}+a_{d-2}^{\floor{\frac{d}{2}}-1}\}A^{d-2{\floor{\frac{d}{2}}}}\\
= &A^{d}-a_{d}^{1} A^{d-2}+(k-1)a_{d}^{2} A^{d-4}-\cdots+(-1)^{\floor{\frac{d}{2}}}(k-1)^{{\floor{\frac{d}{2}}-1}}a_{d}^{\floor{\frac{d}{2}}}A^{d-2{\floor{\frac{d}{2}}}}.\\
\end{split}
\end{equation*}
 Hence the result holds true in this case.\\

\vspace{-1em}
Next, we consider that $g$ is even. The result holds true for $i=0,1,2,\ldots,d-1$, because the intersection numbers agree with those in the case that $g$ is odd. Since $c_{d}=k$, by recurrence relation (\ref{eq2}) we get
\begin{equation*}
\begin{split}
&A A_{d-1}=c_{d}A_{d}+a_{d-1}A_{d-1}+b_{d-2}A_{d-2}
 =k A_{d}+(k-1) A_{d-2},\\
&A_{d}= \frac{1}{k}[A^{d}-a_{d}^{1} A^{d-2}+(k-1)a_{d}^{2} A^{d-4}-\cdots+(-1)^{\floor{\frac{d}{2}}}(k-1)^{{\floor{\frac{d}{2}}-1}}a_{d}^{\floor{\frac{d}{2}}}A^{d-2{\floor{\frac{d}{2}}}}].
\end{split}
\end{equation*}
\end{proof}
In the theorem below we find polynomials of degree $\floor{\frac{g}{2}}$ which give all distance eigenvalues of minimal $(k,g)$-cages when the variable is substituted by adjacency eigenvalues.
\begin{theorem} \label{theorempoly}
If $\lambda$ is an eigenvalue of a minimal $(k,g)$-cage $G$  then $p(\lambda)$ is a distance eigenvalue of $G$ with the same multiplicity as that of $\lambda$, where $p(x)$ is given below.\\
{\small{\begin{equation*}
p(x) = \begin{cases}
\sum\limits_{i=0}^{d-1} [i+\sum\limits_{\substack{j=1\\ d-i\geq 3 }}^{\floor{\frac{d-i-1}{2}}} (-1)^{j} (i+2j)(k-1)^{j-1} a_{i+2j}^{j}]x^{i} +
\frac{d}{k}[A^{d}+\\
\sum\limits_{i=1}^{\floor{\frac{d}{2}}} (-1)^{i} (k-1)^{i-1} a_{d}^{i} x^{d-2i}],&\text{$g$ even}\\
x,&\text{for $g=3$}\\
2x^{2}+x-2k,& \text{for $g=5$}
\end{cases}
\end{equation*}}}

where $a_{i}^{j}$ are as in Lemma \ref{lemmainteger}.
\end{theorem}
\begin{proof} Here we represent $D(G)$ as a polynomial, $p(A)$, of the adjacency matrix $A$ of $G$ and then the theorem follows from a basic result that " If $\lambda$ is an eigenvalue of $A$ then $p(\lambda)$ is an eigenvalue of $p(A)$ with the same multiplicity as that of $\lambda$."
First, we consider that $g$ is an even integer.
From Theorem \ref{theoremai} and equation (\ref{eq1}), the distance matrix $D$ of a minimal $(k,g)$-cage $G$ can be written as below:
\begin{equation*}
\begin{split}
D&=A_{1}+2A_{2}+3A_3\cdots+d A_{d}\\
&=A+2(A^{2}-a_{2}^{1}I)+3(A^{3}-a_{3}^{1} A)+\cdots+(d-1)[A^{d-1}-a_{d-1}^{1} A^{d-3}+\cdots+\\
&(-1)^{\floor{\frac{d-1}{2}}}
(k-1)^{{\floor{\frac{d-1}{2}}}-1} a_{d-1}^{{\floor{\frac{d-1}{2}}}}A^{d-1-2({\floor{\frac{d-1}{2}}})}]+\frac{d}{k}[A^{d}-a_{d}^{1} A^{d-2}+\cdots+(-1)^{\floor{\frac{d}{2}}}\\
&(k-1)^{{\floor{\frac{d}{2}}-1}}a_{d}^{\floor{\frac{d}{2}}}A^{d-2{\floor{\frac{d}{2}}}}]
\end{split}
\end{equation*}
If $d$ is even then $\floor{\frac{d-i-2}{2}}=\floor{\frac{d-i-1}{2}}$ for even $i$ and $\floor{\frac{d-i}{2}}=\floor{\frac{d-i-1}{2}}$ for odd $i$, $i=1,\ldots,d-1$. So we get,
\begin{align*}
D&=[-2a_{2}^{1}+4(k-1)a_{4}^{2}-\cdots+(-1)^{\floor{\frac{d-2}{2}}}(d-2)(k-1)^{\floor{\frac{d-2}{2}}-1}a_{d-2}^{\floor{\frac{d-2}{2}}}]+[1-
3a_{3}^{1}\\
&+\cdots+(-1)^{\floor{\frac{d-1}{2}}}
(d-1)(k-1)^{{\floor{\frac{d-1}{2}}-1}}a_{d-1}^{\floor{\frac{d-1}{2}}}]A+\cdots+[(d-3)-(d-1)a_{d-1}^{1}]\\
&A^{d-3}+(d-2)A^{d-2}+(d-1)A^{d-1}+
\frac{d}{k}[A^{d}-a_{d}^{1} A^{d-2}+(k-1)a_{d}^{2} A^{d-4}-\cdots+\\
&(-1)^{\floor{\frac{d}{2}}}(k-1)^{{\floor{\frac{d}{2}}-1}}a_{d}^{\floor{\frac{d}{2}}}A^{d-2{\floor{\frac{d}{2}}}}]
\end{align*}
\begin{align*}
=&\sum\limits_{i=0}^{d-1} [i+\sum\limits_{\substack{j=1\\ d-i\geq 3 }}^{\ceil{\frac{d-i-1}{2}}} (-1)^{j} (i+2j)(k-1)^{j-1} a_{i+2j}^{j}]A^{i}+\frac{d}{k}[A^{d}+ \sum\limits_{i=1}^{\floor{\frac{d}{2}}} (-1)^{i} (k-1)^{i-1}\\
&a_{d}^{i} A^{d-2i}]
\end{align*}
For $d$ odd, we get
\begin{equation*}
\begin{split}
D =&[-2a_{2}^{1}+4(k-1)a_{4}^{2}-\cdots+
(-1)^{\floor{\frac{d-1}{2}}}(d-1)(k-1)^{{\floor{\frac{d-1}{2}}-1}}a_{d-1}^{\floor{\frac{d-1}{2}}}]+[1-3\\
&a_{3}^{1}+\cdots+(-1)^{\floor{\frac{d-2}{2}}}(d-2)(k-1)^{{\floor{\frac{d-2}{2}}-1}}a_{d-2}^{\floor{\frac{d-2}{2}}}]A+\cdots+[(d-3)-(d-1)\\
&a_{d-1}^{1}]A^{d-3}+(d-2)A^{d-2}+(d-1)A^{d-1}+\frac{d}{k}[A^{d}-a_{d}^{1} A^{d-2}+\cdots+(-1)^{\floor{\frac{d}{2}}}\\
&(k-1)^{{\floor{\frac{d}{2}}-1}}a_{d}^{\floor{\frac{d}{2}}}A^{d-2{\floor{\frac{d}{2}}}}]\\
= &\sum\limits_{i=0}^{d-1} [i+\sum\limits_{\substack{j=1\\ d-i\geq 3 }}^{\floor{\frac{d-i-1}{2}}} (-1)^{j} (i+2j)(k-1)^{j-1} a_{i+2j}^{j}]A^{i} +\frac{d}{k}[A^{d}+ \sum\limits_{i=1}^{\floor{\frac{d}{2}}} (-1)^{i} \\
&(k-1)^{i-1} a_{d}^{i} A^{d-2i}]
\end{split}
\end{equation*} 
In the above, for both $d$ even and odd, $D$ is expressed as a polynomial of the adjacency matrix $A$ of $G$. We take this polynomial as $p(A)$ and obtain the result.\\

If $g$ is an odd integer then by Lemma \ref{lexist}, $g=3$ or $5$. For $g=3$ the value of $d$ is $1$. So,
 by equation (\ref{eq1}) we have $D=p(A)=A$, and so is the result.
 Now for $g=5$, the value of $d$ is $2$. Applying equation (\ref{eq1}) and Theorem \ref{theoremai} we have,
$D=A_{1}+2A_{2}=A+2(A^{2}-a_{2}^{1}I)=A+2(A^{2}-kI)=2A^{2}+A-2kI$,
  and hence the result.
\end{proof}
\begin{theorem} \label{theoremd+1}
Every minimal $(k,g)$-cage, $k\geq 2$, with diameter $d$ has $d+1$ distinct distance eigenvalues.
\end{theorem}
\begin{proof}
For $k=2$ the minimal $(k,g)$-cages are cycles, and the result follows by \cite{Grao}.\\
\noindent
\textbf{Case 1.} In this case we consider that $g$ is an even integer. For $d$ even, applying Lemma \ref{lemmainteger} and Theorem \ref{theorempoly} the distance matrix of the minimal $(k,g)$-cage can be written as,
\begin{align*}
D=&p(A)=A+2(A^{2}-a_{2}^{1}I)+\cdots+(d-2)
[A^{d-2}- a_{d-2}^{1} A^{d-4}+\cdots +(-1)^{\floor{\frac{d-2}{2}}-1}\\
&(k-1)^{\floor{\frac{d-2}{2}}-2}a_{d-2}^{\floor{\frac{d-2}{2}}-1}A^{d-2-2({\floor{\frac{d-2}{2}}-1})}+(-1)^{\floor{\frac{d-2}{2}}}
(k-1)^{{\floor{\frac{d-2}{2}}}-1} a_{d-2}^{{\floor{\frac{d-2}{2}}}}\\
& A^{d-2-2({\floor{\frac{d-2}{2}}})}]+(d-1)
[A^{d-1}- a_{d-1}^{1} A^{d-3}+\cdots +(-1)^{\floor{\frac{d-1}{2}}-1}
(k-1)^{\floor{\frac{d-1}{2}}-2}\\
&a_{d-1}^{\floor{\frac{d-1}{2}}-1}A^{d-1-2({\floor{\frac{d-1}{2}}-1})}+(-1)^{\floor{\frac{d-1}{2}}}
(k-1)^{{\floor{\frac{d-1}{2}}}-1} a_{d-1}^{{\floor{\frac{d-1}{2}}}} A^{d-1-2({\floor{\frac{d-1}{2}}})}]+\\
&\frac{d}{k}[A^{d}-a_{d}^{1} A^{d-2}+\cdots+(-1)^{\floor{\frac{d}{2}}-1}(k-1)^{\floor{\frac{d}{2}}-2}a_{d}^{\floor{\frac{d}{2}}-1}A^{d-2({\floor{\frac{d}{2}}-1})}+(-1)^{\floor{\frac{d}{2}}}\\
&(k-1)^{{\floor{\frac{d}{2}}-1}}a_{d}^{\floor{\frac{d}{2}}}A^{d-2{\floor{\frac{d}{2}}}}]\\
=&A+2(A^{2}-k I)+\cdots+(d-2)
[A^{d-2}-(g_{d-2}^{1}k-h_{d-2}^{1}) A^{d-4}+\cdots +(-1)^{\floor{\frac{d-2}{2}}}\\ 
& k\sum_{i=0}^{\floor{\frac{d-2}{2}}-1}(-1)^{i}\binom{\floor{\frac{d-2}{2}}-1}{i} k^{\floor{\frac{d-2}{2}}-1-i} ]+(d-1)[A^{d-1}- (g_{d-1}^{1}k-h_{d-1}^{1})A^{d-3}\\ &+\cdots +(-1)^{\floor{\frac{d-1}{2}}}\sum_{i=0}^{\floor{\frac{d-1}{2}}-1} (-1)^{i}\binom{\floor{\frac{d-1}{2}}-1}{i} k^{\floor{\frac{d-1}{2}}-1-i} (g_{d-1}^{{\floor{\frac{d-1}{2}}}}k-h_{d-1}^{{\floor{\frac{d-1}{2}}}}) A]+\\
&\frac{d}{k}[A^{d}-(g_{d}^{1}k-h_{d}^{1}) A^{d-2}+\cdots+(-1)^{\floor{\frac{d}{2}}}k \sum_{i=0}^{\floor{\frac{d}{2}}-1}(-1)^{i}\binom{\floor{\frac{d}{2}}-1}{i} k^{\floor{\frac{d}{2}}-1-i} ]\\
=&\frac{1}{k}\big[d\big\{A^{d}+h_{d}^{1}A^{d-2}+\cdots+h_{d}^{\floor{\frac{d}{2}}-1} A^{2}\big\}+\big\{A+2A^{2}+3(A^{3}+h_{3}^{1}A)+\cdots+(d\\ &-1)(A^{d-1}+h_{d-1}^{1}A^{d-3}+\cdots+h_{d-1}^{\floor{\frac{d-1}{2}}}A)-d(g_{d}^{1}A^{d-2}+(g_{d}^{2}+h_{d}^{2})A^{d-4}+(g_{d}^{3}\\ 
&+2h_{d}^{3})A^{d-6}+\cdots+(g_{d}^{\floor{\frac{d}{2}}-1}
+(\floor{\frac{d}{2}}-1)h_{d}^{\floor{\frac{d}{2}}-1})A^{2}+1)\big\}k+\cdots+(-1)^{\floor{\frac{d}{2}}-1}\\ &\{(d-2)+g_{d-1}^{\floor{\frac{d-1}{2}}}(d-1)A-d\}k^{\floor{\frac{d}{2}}}\big],
\end{align*}
since for $d$ even ${{\floor{\frac{d}{2}}}}-1={{\floor{\frac{d-1}{2}}}}={{\floor{\frac{d-2}{2}}}}$. If possible let there be two distinct eigenvalues $\lambda_{i}\neq \lambda_{j}$ of the minimal $(k,g)$-cage such that $p(\lambda_{i})=p(\lambda_{j})$. This implies $k p(\lambda_{i})=k p(\lambda_{j})$ (since $k\neq 0$). Now equating the coefficients of $k^{{\floor{\frac{d}{2}}}}$, we get $(d-2)+(d-1)g_{d-1}^{\floor{\frac{d-1}{2}}}\lambda_{i}-d=(d-2)+(d-1)g_{d-1}^{\floor{\frac{d-1}{2}}}\lambda_{j}-d$, which gives $\lambda_{i}=\lambda_{j}$, a contradiction. This proves that a minimal $(k,g)$-cage has $d+1$ distinct distance eigenvalues for $d$ and $g$ both even.\\

\vspace{-1em}
If $d$ is odd and $g$ is even then by Lemma \ref{lexist} we get that $g$ is equal to $6$. By Lemma \ref{lspec} all distinct eigenvalues of the minimal $(k,6)$-cage are $\pm k, \pm \sqrt{k-1}$. Then the distance matrix of the minimal $(k,6)$-cage can be written as $D=p(A)=\frac{3}{k}A^{3}+2A^{2}+\frac{3-5k}{k}A-2k I$. If possible let there exist two distinct eigenvalues $\lambda_{i}\neq \lambda_{j}$ of minimal $(k,6)$-cage such that $p(\lambda_{i})=p(\lambda_{j})$. Then $k p(\lambda_{i})=k p(\lambda_{j})$ (since $k\neq 0$). This implies 
$-2k^{2}+(2 \lambda_{i}^{2}-5 \lambda_{i})k+3(\lambda_{i}^{3}+\lambda_{i})=-2k^{2}+(2 \lambda_{j}^{2}-5 \lambda_{j})k+3(\lambda_{j}^{3}+\lambda_{j})$.
Equating the coefficients of $k$ we get
$2(\lambda_{i}^{2}-\lambda_{j}^{2})-5(\lambda_{i}-\lambda_{j})=0$, and then $\lambda_{i}+\lambda_{j}=\frac{5}{2}$. Since $k$ is the largest adjacency eigenvalue, $p(k)$ is the largest distance eigenvalue \cite{Ala}. Thus both $\lambda_{i}$ and $\lambda_{j}$ are different from $k$. Then $\lambda_{i},\lambda_{j} \in \{-k, \pm \sqrt{k-1}\}$. Now
 $\lambda_{i}+\lambda_{j}=-k\pm\sqrt{k-1}=\frac{5}{2}$, that is $4k^{2}+16k+29=0$. But this equation does not give any integer solution and since $k$ is an integer, this leads to a contradiction. Hence a minimal $(k,6)$-cage has $4$ distinct distance eigenvalues.\\
 
 \vspace{-1em}
\textbf{Case 2.} In this case we consider that $g$ is an odd integer. A minimal $(k,3)$-cage is a complete graph $K_{n}$ and its distinct distance eigenvalues are $n-1, -1$. Now for $g=5$, all distinct eigenvalues of the minimal $(k,5)$-cage are $k, \frac{-1\pm \sqrt{4k-3}}{2}$. Then from Theorem \ref{theorempoly} the distance matrix of this graph can be written as $D=p(A)=2A^{2}+A-2kI$.
 If possible let there exist two distinct eigenvalues $\lambda_{i}\neq \lambda_{j}$ of minimal $(k,5)$-cage such that $p(\lambda_{i})=p(\lambda_{j})$. This implies $2(\lambda_{i}^{2}-\lambda_{j}^{2})+(\lambda_{i}-\lambda_{j})=0$, and then $\lambda_{i}+\lambda_{j}=-\frac{1}{2}$.
Since $k$ is the largest adjacency eigenvalue, $p(k)$ is the largest distance eigenvalue \cite{Ala}. Thus both $\lambda_{i}$ and $\lambda_{j}$ are different from $k$. Then $\lambda_{i},\lambda_{j}\in \{\frac{-1\pm \sqrt{4k-3}}{2}\}$ and $\lambda_{i}+\lambda_{j}=-1$, a contradiction. Thus a minimal $(k,5)$-cage has $3$ distinct distance eigenvalues.
This proves the theorem.
\end{proof}
 \remark \rm
  Theorem \ref{theoremd+1} supplies a class of graphs to the answer of the problem "Characterize distance regular graphs with diameter $d$ and having exactly $d+1$ distinct $D$-eigenvalues", asked by Atik and Panigrahi \cite{Atik1}.\\
 
 \vspace{-1em}
  A minimal $(k,3)$-cage is  a complete graph and its distance spectrum is mentioned in Theorem \ref{theoremd+1}. Minimal $(k,4)$-cages and $(2,g)$-cages are complete bipartite graphs and cycles respectively, and their distance spectrum can be found in \cite{Stev} and \cite{Grao}. So in the theorem below we present the distance spectrum of minimal $(k,g)$-cages, $k\geq 3$ and $g \geq 5$, by applying Lemma \ref{lspec} and Theorems \ref{theoremai} and \ref{theorempoly}. 
  
  \begin{theorem} \begin{enumerate}
    \item The distance matrix of a minimal $(k,5)$-cage is $D=-2k I+A+2A^{2}$.
    \begin{enumerate}
        \item The distance matrix of the minimal $(3,5)$-cage (Petersen graph) is $D=-6I+A+2A^{2}$, and its distance spectrum is $\{15,-3^{(5)},0^{(4)}\}$. 
        \item The distance matrix of the minimal $(7,5)$-cage (Hoffman-Singleton graph) is $D=-14I+A+2A^{2}$, and its  distance spectrum is $\{91, -4^{(28)}, 1^{(21)}\}$.
        \item The distance matrix of the minimal $(57,5)$-cage (if exists) is $D=-114I+A+2A^{2}$, and its distance spectrum is $\{6441, -9^{(1729)}, 6^{(1520)}\}$.
    \end{enumerate}
    \item The distance matrix of a minimal $(k,6)$-cage is $D=\frac{3}{k} A^{3}+2A^{2}-\frac{5k-3}{k} A-2k I$, and its distance spectrum is $\{5k^2-7k+3, -k^2+3k-3, (-2(1+\sqrt{k-1}))^{(m_{\sqrt{k-1}})}, (-2(1-\sqrt{k-1}))^{(m_{-\sqrt{k-1}})}\}$, where $m_{\pm \sqrt{k-1}}=\frac{n k(k-1)}{2 (k^2-k+1)}$.\\
    (Heawood graph is a minimal $(3,6)$-cage and its distance matrix is $D=A^{3}+2A^{2}-4A-6I$, and its distance spectrum is $\{27, -3,  (-2(1+\sqrt{2}))^{(6)},\\ (-2(1-\sqrt{2}))^{(6)}\}$).
    \item The distance matrix of a minimal $(k,8)$-cage is $D=\frac{4}{k} A^{4}+3A^{3}+\frac{8-10k}{k} A^{2}-(6k-4)A+(2k-4) I$, and its distance spectrum is $\{7k^3-16k^2+14k-4, k^3-4k^2+6k-4, (2k-4)^{(m_{0})}, (-2(k+\sqrt{2(k-1)}))^{(m_{\sqrt{2(k-1)}})}, (-2(k-\sqrt{2(k-1)}))^{(m_{-\sqrt{2(k-1)}})} \}$, where $m_{0}=\frac{n(k-1)}{2k}$, $m_{\pm \sqrt{2(k-1)}}=\frac{n k(k-1)}{4(k^2-2k+2)}$.\\
    ( Levi graph is a minimal $(3,8)$-cage and its distance matrix is $D=\frac{4}{3} A^{4}+3A^{3}-\frac{22}{3} A^{2}-14A+2 I$, and its distance spectrum is $\{83, 5,2^{(10)},\\ -10^{(9)}, -2^{(9)}\}$).
    \item The distance matrix of a minimal $(k,12)$-cage is $D=\frac{6}{k} A^{6}+5A^{5}+\frac{24-26k}{k} A^{4}-(20k-18)A^{3}+\frac{24k^{2}-44k+18}{k}A^{2}+(15k^{2}-26k+9)A-(2 k^{2}-6k+6) I$, and its distance spectrum is $\{11k^5-46 k^4+81k^3-72 k^2+33k-6,~k^5-6k^4+15 k^3-20 k^2+15k-6,~(-2k^2+6k-6)^{(m_{0})},~(2(k-2)(k+\sqrt{k-1}))^{(m_{\sqrt{k-1}})},~(2(k-2)(k-\sqrt{k-1}))^{(m_{-\sqrt{k-1}})},~(-2k(k+\\\sqrt{3(k-1)}))^{(m_{\sqrt{3(k-1)}})},(-2k(k-\sqrt{3(k-1)}))^{(m_{-\sqrt{3(k-1)}})} \}$, where $m_{0}=\frac{n(k-1)}{3k}$, $m_{\pm \sqrt{k-1}}=\frac{n k(k-1)}{4(k^2-k+1)}$, and $m_{\pm\sqrt{3(k-1)}}=\frac{n k(k-1)}{12(k^2-3k+3)}$.
\end{enumerate}
\end{theorem}
\section{Distance spectrum of some distance biregular graphs}
In the next theorem we show that every DBR graph is a $2$-partitioned transmission regular graph.

\begin{theorem} \label{theoremdbr}
All distance biregular (DBR) graphs are $2$-partitioned transmission regular graphs.
\end{theorem}
\begin{proof}
Let $G$ be a DBR graph with partite sets $V_{1}$ and $V_{2}$. 
Each vertex $u$ in $V_{1}$ has $l_{i}$, a constant, number of vertices at distance $i$. From $u$, even distance vertices are situated in $V_{1}$ and odd distance vertices are in $V_{2}$. Thus the number of vertices of even and odd distances is constant from each vertex $u\in V_{1}$.  So we get, $q_{11}=\sum\limits_{v\in V_{1}} d(u,v)=\sum\limits_{i=0}^{\floor{\frac{d_{1}}{2}}} 2i l_{2i}$ and $q_{12}=\sum\limits_{v\in V_{2}} d(u,v)=\sum\limits_{i=0}^{\floor{\frac{d_{1}}{2}}} (2i+1) l_{2i+1}$ are constants, where $d_{1}=$max $\{d(x,y):x\in V_{1},y\in V(G)\}$. Similarly the sum of distances from each $w \in V_{2}$, $q_{21}=\sum\limits_{v\in V_{1}} d(w,v)=\sum\limits_{i=0}^{\floor{\frac{d_{2}}{2}}} 2i l_{2i}^{\prime}$ and $q_{22}=\sum\limits_{v\in V_{2}} d(w,v)=\sum\limits_{i=0}^{\floor{\frac{d_{2}}{2}}} (2i+1) l_{2i+1}^{\prime}$ are constants, where $l_{i}^{\prime}$ is the number of vertices at distance $i$ from $w$ and $d_{2}=$max $\{d(x,y):x\in V_{2},y\in V(G)\}$. Hence the result.
\end{proof}
The larger root of the  quotient matrix $Q=
 \begin{pmatrix}
 q_{11} & q_{12}\\
 q_{21} & q_{22}
 \end{pmatrix}
 $ is the distance spectral radius of any DBR graphs by Lemma \ref{q2}.
The subdivision graph $S(G)$ of a minimal $(k,g)$-cage $G$ is a DBR graph \cite{Mohar}. So applying Theorem \ref{theoremdbr} we determine distance spectral radius of these graphs. Unless otherwise stated, in the remaining of the paper $(V_{1},V_{2})$ is taken as the vertex partition of subdivision of a minimal $(k,g)$-cage $G$, where $V_{1}=V(G)$ and $V_{2}$ is the set of all new vertices inserted on edges of $G$.

\begin{theorem} \label{theoremdsr}
Let $g$ be an even integer and $G$ be a minimal $(k,g)$-cage  with diameter $d$. The distance spectral radius of the subdivision graph $S(G)$ is\\
 $(3k-2)S_{1}^{\prime}+d k(k-1)^{d-1}+\\
 \sqrt{(k-2)^2 S_{1}^{\prime 2}+ 2k S_{ 2}^{\prime 2}+2d (k-2)^{2} (k-1)^{d-1} S_{1}^{\prime}+d^{2} (k-2)^{2} (k-1)^{2d-2}}$, where $S_{1}^{\prime}= \frac{1}{(2-k)^{2}} 
 [(k d-2d-k+1)(k-1)^{d-1}+1]$ and $S_{2}^{\prime}=\frac{1}{(2-k)^{2}} [(2k d-4d-k)(k-1)^{d}+k]$.
\end{theorem}
\begin{proof}
We know that \cite{Mohar} the intersection array of any vertex $u\in V_{1}$ is
$\{k,1 , k-1 , 1 , \ldots , k-1 ,1;1 , 1 , 1 , 1,\ldots, 1 , k\}$ and the intersection array of any vertex  $v\in V_{2}$ is $\{2 , k-1 , 1 , k-1 , \cdots, 1 ,k-1;1 , 1 , 1 , 1, \cdots , 1 , 2\}$. Since $S(G)$ is obtained by inserting a new vertex in every edge and diameter of $G$ is $d=\floor{\frac{g}{2}}$, applying Theorem \ref{theoremdbr} we have $l_{0}=1$, $l_{2i}=l_{2i-1}=k(k-1)^{i-1}$, $i=1,2,\ldots,d-1$, $l_{2d-1}= k(k-1)^{d-1}$, $l_{2d}= (k-1)^{d-1}$, $l_{1}^{\prime}=2$, $l_{2i}^{\prime}=l_{2i+1}^{\prime}=2(k-1)^{i}$, $i=1,2,\ldots,d-1$, and $l_{2d}^{\prime}=(k-1)^{d}$.
Thus \begin{align*}
     q_{11}&=\sum_{i=0}^{d} 2i l_{2i}=0+2k+4k(k-1)+\cdots+(2d-2)k(k-1)^{d-2}+2d (k-1)^{d-1}
     \\
   &=2k[1+2(k-1)+3(k-1)^{2}+\cdots+(d-1)(k-1)^{d-2}]+2d(k-1)^{d-1}\\
   &=2 k S_{1}^{\prime}+2d(k-1)^{d-1},
   \end{align*}
   
  \newpage 
where $S_{1}^{\prime}=1+2(k-1)+3(k-1)^{2}+\cdots+(d-1)(k-1)^{d-2}=\frac{1}{(2-k)^{2}} [(k d-2d-k+1)(k-1)^{d-1}+1]$.
\begin{align*}
q_{12}=&\sum_{i=0}^{d-1} (2i+1) l_{2i+1} =k+3k(k-1)+5k(k-1)^{2}+\cdots+(2d-1) k(k-1)^{d-1}\\
 =&k[1+3(k-1)+5(k-1)^{2}+\cdots+(2d-1)(k-1)^{d-1}]=k S_{2}^{\prime},
 \end{align*}
 where
 $S_{2}^{\prime}=1+3(k-1)+5(k-1)^{2}+\cdots+(2d-1)(k-1)^{d-1}=\frac{1}{(2-k)^{2}} [(2k d-4d-k)(k-1)^{d}+k]$.
 \begin{align*}
q_{21}=&\sum_{i=0}^{d-1} (2i+1) l_{2i+1}^{\prime}=2+3\times 2(k-1)+5\times 2(k-1)^{2}+\cdots+(2d-1)\times 2\\
&(k-1)^{d-1}= 2 S_{2}^{\prime}.\\
q_{22}=&\sum_{i=0}^{d} 2i l_{2i}^{\prime}=2 \times 2(k-1)+4\times 2(k-1)^{2}+6 \times 2(k-1)^{3}+\cdots+(2d-2)\\
&\times 2(k-1)^{d-1}+2d\times (k-1)^{d}=4(k-1) S_{1}^{\prime}+2d (k-1)^{d}.
 \end{align*}
 Thus $Q=
 \begin{pmatrix}
 q_{11} & q_{12}\\
 q_{21} & q_{22}
 \end{pmatrix}
 $ is a quotient matrix of the distance matrix of $S(G)$ when $g$ is even. The characteristic polynomial of $Q$ is $x^{2}-\{2(3k-2)S_{1}^{\prime}+2d k(k-1)^{d-1}\}x+8k(k-1)S_{1}^{^{\prime}2}-2k S_{2}^{\prime 2}+4d (k+2)(k-1)^{d} S_{1}^{\prime}+4 d^{2} (k-1)^{2d-1}=0$, and its larger root
 $(3k-2)S_{1}^{\prime}+d k(k-1)^{d-1}+\\
 \sqrt{(k-2)^2 S_{1}^{\prime 2}+ 2k S_{ 2}^{\prime 2}+2 d (k-2)^{2} (k-1)^{d-1} S_{1}^{\prime}+d^{2} (k-2)^{2} (k-1)^{2d-2}}$, where $S_{1}^{\prime}=\frac{1}{(2-k)^{2}}[(k d-2d-k+1)(k-1)^{d-1}+1]$ and $S_{2}^{\prime}=\frac{1}{(2-k)^{2}} [(2k d-4d-k)(k-1)^{d}+k]$ is the distance spectral radius of $S(G)$ by Lemma \ref{q2}.
\end{proof}
By Lemma \ref{lexist}, if $g$ is an odd integer then minimal $(k,g)$-cages exist only for $g=3$ and $5$. So in the next theorem we determine distance spectral radius of subdivision of minimal $(k,g)$-cages for these two cases only. 
\begin{theorem} \label{theoremgodd}
Let $g$ be an odd integer and $G$ be a minimal $(k,g)$-cage. The distance spectral radius of the subdivision graph $S(G)$ is,\\ 
{\small{$\lambda_{1}(S(G))=
 \begin{cases}
 \frac{1}{2} [2k^{2}+
 \sqrt{2k(2k+1)(k^{2}+1)}], &\text{if $g=3$}\\
 \frac{1}{2} [(3k^{3}+k-2)+
 \sqrt{9k^{6}+2k^{5}+14k^{4}-40k^{3}+41k^{2}-18k+4}], &\text{if $g=5$}
 \end{cases}$}}
\end{theorem}
\begin{proof}
First, let $g$ be equal to $3$.
We know that \cite{Mohar} intersection array of any vertex in $V_1$ is
$\{k,1 , k-1;1 , 1, 2\}$ and the intersection array of any vertex in $V_{2}$ is $\{2 , k-1 , 1 , k-2;1 , 1 , 2 , 2\}$. By Lemma \ref{ldbr} we have, $l_{0}=1$, $l_{1}=l_{2}=k$, $l_{3}=\frac{1}{2}k(k-1)$, $l_{0}^{\prime}=1$, $l_{1}^{\prime}=2$, $l_{2}^{\prime}=2(k-1)$, $l_{3}^{\prime}=(k-1)$, and $l_{4}^{\prime}=\frac{1}{2}(k-1)(k-2)$.
Thus $q_{11}=2k$, $q_{12}=k+\frac{3}{2}k(k-1)=\frac{1}{2}k(3k-1)$, $q_{21}=2+3(k-1)=3k-1$, $q_{22}=2\times 2(k-1)+4\times \frac{1}{2}(k-1)(k-2)=2k(k-1)$.\\
So $Q=
 \begin{pmatrix}
 q_{11} & q_{12}\\
 q_{21} & q_{22}
 \end{pmatrix}
 $ is a quotient matrix of the distance matrix of $S(G)$. The characteristic polynomial of $Q$ is $x^{2}-2k^{2}x-\frac{1}{2}k (k+1)^{2}=0$, and its larger root $\frac{1}{2} [2k^{2}+
 \sqrt{2k(2k+1)(k^{2}+1)}]$ is the distance spectral radius of $S(G)$ by Lemma \ref{q2}.\\
Next we take $g=5$. Intersection array of any vertex in $ V_1$ is
$\{k,1,k-1,1,k-1;1,1,1,1, 2\}$ and the intersection array of any vertex in $V_2$ is $\{2,k-1,1,k-1,1,k-2;1,1,1,1, 2,2\}$. By Lemma \ref{ldbr} we have, $l_{0}=1$, $l_{1}=l_{2}=k$, $l_{3}=l_{4}=k(k-1)$, $l_{5}=\frac{1}{2}k(k-1)^{2}$, $l_{0}^{\prime}=1$, $l_{1}^{\prime}=2$, $l_{2}^{\prime}=l_{3}^{\prime}=2(k-1)$, $l_{4}^{\prime}=2(k-1)^{2}$, $l_{5}^{\prime}=(k-1)^{2}$, and $l_{6}^{\prime}=\frac{1}{2}(k-1)^{2}(k-2)$. Thus $q_{11}=2k+4k(k-1)=2k(2k-1)$, $q_{12}=k+3k(k-1)+\frac{5}{2}k(k-1)^{2}=\frac{1}{2}k(5k^{2}-4k+1)$, $q_{21}=2+3\times 2(k-1)+5\times(k-1)^{2}=(5k^{2}-4k+1)$, $q_{22}=2\times 2(k-1)+4\times 2(k-1)^{2}+6\times \frac{1}{2}(k-1)^{2}(k-2)=(k-1)(3k^{2}-k+2)$.\\
So $Q=
 \begin{pmatrix}
 q_{11} & q_{12}\\
 q_{21} & q_{22}
 \end{pmatrix}
 $ is a quotient matrix of the distance matrix of $S(G)$. The characteristic polynomial of $Q$ is $x^{2}-(3k^{3}+k-2)x-\frac{1}{2}k (k^{4}+4k^{3}-14k^{2}+20k-7)=0$, and its larger root\\ $\frac{1}{2} [(3k^{3}+k-2)+
 \sqrt{9k^{6}+2k^{5}+14k^{4}-40k^{3}+41k^{2}-18k+4}]$ is the distance spectral radius of $S(G)$ by Lemma \ref{q2}.
\end{proof}

\example \rm
We know that the Heawood graph is the minimal $(3,6)$-cage. By Theorem \ref{theoremdsr}, $Q=\begin{pmatrix}
54 & 81\\
54 & 88
\end{pmatrix}$ is a quotient matrix of the distance matrix of subdivision of Heawood graph, and its characteristic polynomial is $x^{2}-142x+378$. So the distance spectral radius of subdivision of Heawood graph is $71+\sqrt{4663}$. We also compute the distance characteristic polynomial of subdivision of Heawood graph, which is $(x^{2}-142x+378)(x+2)^{8}(x+6)(x^{4}+20x^{3}-60x^{2}-80x+112)^{6}$. So the D-spectrum of subdivision of Heawood graph is the union of $\{71\pm \sqrt{4663}, -2^{(8)}, -6\}$ and the set of roots of the polynomial $(x^{4}+20 x^{3}-60 x^{2}-80 x+112)^{6}$.

\example \rm
We know that the Petersen graph is the minimal $(3,5)$-cage. By Theorem \ref{theoremgodd}, $Q=\begin{pmatrix}
30 & 51\\
34 & 52
\end{pmatrix}$ is a quotient matrix of the distance matrix of subdivision of Petersen graph, and its characteristic polynomial is $x^{2}-82x-174$. So the distance spectral radius of subdivision of Petersen graph is $41+\sqrt{1855}$. We also compute the distance characteristic polynomial of subdivision of Petersen graph, which is $(x^{2}-82x-174)(x^{2}+16x-4)^{5}(x^{2}-2x-4)^{4}(x+2)^{5}$. So the D-spectrum of subdivision of minimal $(3,5)$-cage is
 $\{41\pm \sqrt{1855},(-8\pm 2\sqrt{17})^{(5)}, (1\pm \sqrt{5})^{(4)}, -2^{(5)}\}$.\\

\vspace{-1em}
In the next two theorems we find distance spectrum of subdivision of minimal $(k,3)$-cages (complete graphs $K_{k+1}$) and $(k,4)$-cages (complete bipartite graphs $K_{k,k}$). We denote the $m\times n$ all one matrix by $J_{m\times n}$ (or simply by $J$ if its order is clear from the context) and an $n$-dimensional all one vector by $1_{n}$.
\begin{theorem}
The distance spectrum of subdivision of a minimal $(k,3)$-cage is $\{k^{2}\pm \sqrt{\frac{1}{2}k(k^2 +1)(2k+1)},~(-2k)^{(k)},0^{({k+1\choose 2}-1)}\}$.
\end{theorem}
\begin{proof}
The block matrix representation of $D(S(K_{k+1}))$ with respect to the bipartition $V_1 \cup V_2$ of $S(K_{k+1})$ is given by\\
{\small{$D(S(K_{k+1}))=\begin{bmatrix}
2(J_{k+1 \times k+1}-I_{k+1 \times k+1}) & 3J_{k+1 \times {k+1\choose 2}}-2R_{k+1 \times {k+1\choose 2}}\\
3J_{{k+1\choose 2}\times k+1}-2R_{{k+1\choose 2}\times k+1}^{T} & 4J_{{k+1\choose 2} \times {k+1\choose 2}}-2R_{{k+1\choose 2} \times {k+1}}^{T} R_{k+1 \times {k+1\choose 2}}\\
\end{bmatrix}$}}, where $J$ is the all one matrix and $R$ is the incidence matrix of $K_{k+1}$.
The adjacency spectrum of $K_{k+1}$ is $\{k, -1^{(k)}\}$.
Let $X$ be an eigenvector of $A(K_{k+1})$ corresponding to the eigenvalue $-1$. So $X$ is orthogonal to the all one vector. Also by Lemma \ref{l1}, $RR^{T}=A(K_{k+1})+kI$. Thus
 \[\begin{bmatrix}
2(J-I) & 3J-2R\\
3J-2R^{T} & 4J-2R^{T} R\\
\end{bmatrix} \begin{bmatrix}
X\\
R^{T} X
\end{bmatrix}
=\begin{bmatrix}
-2X-2RR^{T}X\\
-2R^{T}X-2R^{T} RR^{T}X\\
\end{bmatrix}\]
\[=\begin{bmatrix}
-2X-2(k-1)X\\
-2R^{T}X-2(k-1)R^{T}X\\
\end{bmatrix} 
=-2k\begin{bmatrix}
X\\
R^{T}X\\
\end{bmatrix}.\]\\
So $-2k$ is an eigenvalue of $D(S(K_{k+1}))$ with multiplicity $k$.\\
Let $Y$ be an eigenvector of $J_{{k+1\choose 2}\times {k+1 \choose 2}}$ corresponding to the eigenvalue $0$ with multiplicity ${k+1 \choose 2}-1$. Then $Y$ is orthogonal to the all one vector.
 Now \[\begin{bmatrix}
2(J-I) & 3J-2R\\
3J-2R^{T} & 4J-2R^{T} R\\
\end{bmatrix} \begin{bmatrix}
RY\\
-Y\\
\end{bmatrix}
=\begin{bmatrix}
-2RY+2RY\\
-2R^{T}RY+2R^{T}RY\\
\end{bmatrix} 
=0\begin{bmatrix}
RY\\
-Y\\
\end{bmatrix}.\]\\
 So $0$ is an eigenvalue of $D(S(K_{k+1}))$ with multiplicity ${k+1\choose 2}-1$. Now the eigenvectors $\begin{bmatrix}
X\\
R^{T}X
\end{bmatrix}$ and $\begin{bmatrix}
RY\\
-Y
\end{bmatrix}$ of $D(S(K_{k+1}))$ are orthogonal to $\begin{bmatrix}
1_{k+1}\\
0
\end{bmatrix}$ and $\begin{bmatrix}
0\\
1_{k+1\choose 2}
\end{bmatrix}$ respectively. So every other eigenvector $Z$ of $D(S(K_{k+1}))$ is of the form $\begin{bmatrix}
a 1_{k+1}\\
b 1_{{k+1 \choose 2}}
\end{bmatrix}$, $a,b\neq 0$. Now $D(S(K_{k+1}))Z=\lambda Z$ implies,\begin{align*}
    &2k a + \frac{1}{2} k (3k-1)b=\lambda a,\\
    &(3k-1)a + 2 k(k-1) b=\lambda b
\end{align*}
Since $a,b \neq 0$, solving the above equations we get $\lambda^{2}-2k^{2}\lambda-\frac{1}{2}k (k+1)^{2}=0$, and hence the result.
\end{proof}
\begin{theorem}
  The distance spectrum of subdivision of a minimal $(k,4)$-cage is
$\{2k^2 +k-2 \pm \sqrt{(4k^4 - 2k^3 + 9k^2 - 12k + 4)}, 2k-4,~ 0^{((k-1)^{2})},(-(k+2)+ \sqrt{k^{2}+4})^{(2k-2)}, (-(k+2)- \sqrt{k^{2}+4})^{(2k-2)}\}$.
\end{theorem}
\begin{proof}
We take the vertex partition of $S(K_{k,k})$ as $V_{1} \cup V_{2} \cup V_{3}$, where $(V_{1},~V_{2})$ is bipartition of $K_{k,k}$ and $V_{3}$ is the set of all new vertices inserted on edges of $K_{k,k}$. The distance matrix of $S(K_{k,k})$ can be written as,
{\footnotesize{$D(S(K_{k,k}))=\begin{bmatrix}
4(J_{k\times k}-I_{k \times k}) & 2J_{k \times k} & (3J_{k \times k}-2I_{k \times k})\otimes 1_{k}^{T}\\
2J_{k \times k} & 4(J_{k \times k}-I_{k \times k}) & 1_{k}^{T} \otimes (3J_{k \times k}-2I_{k \times k})\\
(3J_{k \times k}-2I_{k \times k})\otimes 1_{k} & 1_{k}  \otimes (3J_{k \times k}-2I_{k \times k}) & 4(J_{{k^2 \times k^2}}- I_{{k^2 \times k^2}})-2A(L(K_{k,k}))_{k^{2} \times k^2}\\
\end{bmatrix}$.}}\\ 
Adjacency matrix $A(K_{k, k})$ and incidence matrix $R$ of $K_{k,k}$ are $\begin{bmatrix}
0 & J_{k\times k}\\
J_{k \times k} & 0\\
\end{bmatrix}$ and $\begin{bmatrix}
I_{k \times k}\otimes 1_{k}^{T}\\
1_{k}^{T}\otimes I_{k\times k}
\end{bmatrix}$ respectively. Thus $3J_{k \times k^{2}}-2R_{k\times k^{2}}=\begin{bmatrix}
(3J_{k\times k}-2I_{k \times k})\otimes 1_{k}^{T}\\
1_{k}^{T} \otimes (3J_{k\times k}-2I_{k \times k})
\end{bmatrix}$.
Let $X$ be an eigenvector of $A(L(K_{k,k}))$ corresponding to the eigenvalue $-2$ with multiplicity $(k-1)^{2}$. 
Applying Lemma \ref{l2} we have $RX=0$. So $X$ is orthogonal to the all one vector.
Now \[\begin{bmatrix}
4(J-I) & 2J & (3J-2I)\otimes 1_{k}^{T}\\
2J & 4(J-I) & 1_{k}^{T} \otimes (3J-2I)\\
(3J-2I)\otimes 1_{k} & 1_{k} \otimes (3J-2I) & 4(J- I)-2A(L(K_{k, k}))\\
\end{bmatrix} \begin{bmatrix}
0\\
0\\
X\\
\end{bmatrix}\]\\
$~~~~~~~~~~~~~~~=\begin{bmatrix}
((3J-2I)\otimes 1_{k}^{T})X\\
(1_{k}^{T} \otimes (3J-2I)) X\\
-4X+4X\\
\end{bmatrix}$=$\begin{bmatrix}
(3J-2R)X\\
0\\
\end{bmatrix}$
$=0\begin{bmatrix}
0\\
0\\
X\\
\end{bmatrix}$.\\
Thus $0$ is an eigenvalue of $D(S(K_{k,k}))$ with multiplicity $(k-1)^{2}$.\\
Let $Z$ be an eigenvector of $A(K_{{k, k}})$ corresponding to the eigenvalue $0$ with multiplicity $2k-2$. Also let $X^{\prime}$ and $Y$ be vectors orthogonal to the all one vector $1_{k}$. If $\begin{bmatrix}
X^{\prime}\\
Y\\
Z\\
\end{bmatrix}$ happens to be an eigenvector of $D(S(K_{k, k}))$ corresponding to an eigenvalue $\lambda$, then it must satisfy,
\[\begin{bmatrix}
4(J-I) & 2J & (3J-2I)\otimes 1_{k}^{T}\\
2J & 4(J-I) & 1_{k}^{T} \otimes (3J-2I)\\
(3J-2I)\otimes 1_{k} & 1_{k} \otimes (3J-2I) & 4J-2R^{T} R\\
\end{bmatrix} \begin{bmatrix}
X^{\prime}\\
Y\\
Z\\
\end{bmatrix}=\lambda\begin{bmatrix}
X^{\prime}\\
Y\\
Z\\
\end{bmatrix}.\]\\
This implies,
\begin{align*}
& -4X^{\prime}+((3J-2I)\otimes 1_{k}^{T})Z=\lambda X^{\prime}\\ 
&-4Y+(1_{k}^{T} \otimes (3J-2I))Z=\lambda Y\\
&\{(3J-2I)\otimes 1_{k}\}X^{\prime}+\{1_{k} \otimes (3J-2I)\}Y-2 R^{T} R Z =\lambda Z.
\end{align*} 
Let $W=\begin{bmatrix}
X^{\prime}\\
Y
\end{bmatrix}$. Combining the relations we get,
\begin{align*}
&-4W+(3J-2R)Z=\lambda Z,\\
& (3J^{T}-2R^{T})W-2R^{T}RZ=\lambda Z.
\end{align*}
From the first equation we get,
 $-4W-2RZ=\lambda Z,~ RZ=-\frac{1}{2} (\lambda+4)W$.
Applying Lemma \ref{l1} in the second equation we get
\begin{align*}
 &-2R R^{T}W-2R R^{T}RZ=\lambda RZ,~
 -2(A+k I)W-2(A+k I)RZ=\lambda RZ,\\
 &-2kW+k(\lambda+4)W=-\frac{1}{2}\lambda (\lambda+4)W.
\end{align*}
Thus $\lambda^{2}+(2k+4)\lambda+4k=0$. So $-(k+2)\pm \sqrt{k^{2}+4}$ are two eigenvalues of $D(S(K_{k, k}))$ with multiplicity $2k-2$. Now $\begin{bmatrix}
0\\
0\\
X
\end{bmatrix}$ is orthogonal to $\begin{bmatrix}
0\\
0\\
1_{k^{2}}
\end{bmatrix}$ and $\begin{bmatrix}
X^{\prime}\\
Y\\
Z
\end{bmatrix}$ is orthogonal to both $\begin{bmatrix}
1_{k}\\
0\\
0
\end{bmatrix}$ and $\begin{bmatrix}
0\\
1_{k}\\
0
\end{bmatrix}$. So every other eigenvector $U$ of $D(S(K_{k, k}))$ is of the form $\begin{bmatrix}
a1_{k}\\
b1_{k}\\
c1_{k^{2}}\\
\end{bmatrix}$, $a,b,c\neq 0$. Now $D(S(K_{k\times k}))U=\mu U$ implies,
\begin{align*}
   & 4(k-1)a+ 2 k b + k(3k-2) c=\mu a\\
   & 2k a+ 4(k-1)b+ k(3k-2) c=\mu b\\
   & (3k-2) a+ (3k-2)b+ 4k(k-1)c=\mu c.
\end{align*}
Since $a,b,c \neq 0$, solving the above equations we get
 $\mu^{3}- (4k^{2} + 4k - 8)\mu^{2} +( 14k^{3} - 28k^{2} - 8k + 16) \mu - (12k^{4} - 56k^{3}+ 80k^{2} -32k)=0 $. So $(\mu-2k+4)(\mu^{2}-(4k^{2}+2k-4)\mu+(6k^3-16k^2+8k)=0$, and hence the result.
\end{proof}
\remark \rm
Theorem \ref{theoremdsr} also gives that the distance spectral radius of subdivision of minimal $(k,4)$-cage is $2k^2 +k-2 + \sqrt{(4k^4 - 2k^3 + 9k^2 - 12k + 4)}$.

\section{\textbf{Concluding Remarks}}
It is known that a distance regular graph of diameter $d$ has exactly $d+1$ distinct eigenvalues. However this is not the case for distance eigenvalues. The authors in  \cite{Atik1} proved that every distance regular graph of diameter $d$ has at the most $d+1$ distinct distance eigenvalues and asked for characterization of distance regular graphs which have exactly $d+1$ distinct distance eigenvalues. In this paper we proved that all minimal cages have exactly $d+1$ distinct distance eigenvalues. We also found distance spectral radius of DBR graphs and determined the full distance spectrum for some DBR graphs associated with minimal $(k,g)$-cages. For the remaining, the following matrix representation of distance matrix of subdivision of a minimal $(k,g)$-cage $G$ may be useful. Here we consider $(V_{1},V_{2})$ as the vertex partition of subdivision of a minimal $(k,g)$-cage $G$, where $V_{1}=V(G)$ and $V_{2}$ is the set of all new vertices inserted on edges of $G$.\\ 

So for $g$ even, $D(S(G))=\begin{bmatrix}
2D(G) & \frac{1}{2}D(G) R(G)\\
\frac{1}{2}R(G)^{T} D(G)^{T} & 2D(L(G))\\.
\end{bmatrix}$, and
for $g$ odd,\\ $D(S(G))=\begin{bmatrix}
2D(G) & \frac{1}{2}D(G) R(G)+E\\
\frac{1}{2}R(G)^{T} D(G)^{T}+E^{T} & 2D(L(G))\\.
\end{bmatrix}$, where $R(G)$, $D(G)$ and $D(L(G))$ are incidence matrix, distance matrix and distance matrix of the line graph of $G$ respectively. $E$ is a matrix whose rows are indexed by vertices of $G$ and columns are indexed by vertices on $V_{2}$ and $(i,j)^{th}$ entry of $E$ is $1$ if $d(v_{i},u_{j})$= max $\{d(v,u):v\in V_{1}, u\in V_{2}\}$ and $0$ otherwise.

\section{Acknowledgement}
The first author is grateful to Council of Scientific and Industrial Research (CSIR), India [Grant number: $09/081(1283)/2016-EMR-I$], for funding the research.


\begin{thebibliography}{30}
\bibitem{Ala} A. Alazemi et al. "Distance-regular graphs with small number of distinct distance eigenvalues." \emph{Linear Algebra and its Applications} 531 (2017): 83-97.

\bibitem{Ali} G. Aalipour et al. "On the distance spectra of graphs." \emph{Linear Algebra and its Applications} 497 (2016): 66-87.

\bibitem{Aou} M. Aouchiche, and P. Hansen. "Distance spectra of graphs: A survey." \emph{Linear algebra and its applications} 458 (2014): 301-386.

\bibitem{Atik} F. Atik, and P. Panigrahi, "Distance Spectral Radius of Some k-partitioned Transmission Regular Graphs." \emph{Conference on Algorithms and Discrete Applied Mathematics.} Springer, Cham, 2016.

\bibitem{Atik1} F. Atik, and P. Panigrahi, "On the distance spectrum of distance regular graphs." \emph{Linear Algebra and its Applications} 478 (2015): 256-273.

\bibitem{Big} N. Biggs, N. L. Biggs, and B. Norman, "Algebraic graph theory." Vol. 67. \emph{Cambridge university press}, 1993.

\bibitem{Bro} A. E. Brouwer,W.H. Haemers, Spectra of Graphs. \emph{Springer}, New York (2011).

\bibitem{BroDR} A. E. Brouwer, A. M. Cohen, and A. Neumaier. "Distance-Regular Graphs." (1989).

\bibitem{Bala} K. Balasubramanian, "Computer generation of distance polynomials of graphs." \emph{Journal of Computational Chemistry} 11.7 (1990): 829-836.

\bibitem{Cve} D. Cvetkovic, M. Doob, and H. Sachs. "Spectra of Graphs-Theory and Application. 1980." \emph{Pure and Applied Mathematics} (1980).

\bibitem{Exo} G. Exoo, and R. Jajcay, "Dynamic cage survey." \emph{The electronic journal of combinatorics} (2012): DS16-July.

\bibitem{God} C.D. Godsil, and J. Shawe-Taylor. "Distance-regularised graphs are distance-regular or distance-biregular." \emph{Journal of Combinatorial Theory}, Series B 43.1 (1987): 14-24.

\bibitem{Gra} R.L. Graham, and H.O. Pollak, "On the addressing problem for loop switching." \emph{The Bell System Technical Journal} 50.8 (1971): 2495-2519.

\bibitem{Grao} A. Graovac, G. Jashari, and M. Strunje. "On the distance spectrum of a cycle." \emph{Aplikace matematiky} 30.4 (1985): 286-290.

\bibitem{Lin} H. Lin, et al. "A survey on distance spectra of graphs." \emph{Adv. Math.(China)} 50.01 (2021): 29-76.

\bibitem{Mohar} B. Mohar and J. Shawe-Taylor, "Distance-biregular graphs with 2-valent vertices and distance-regular line graphs." \emph{Journal of Combinatorial Theory, Series B} 38.3 (1985): 193-203.

\bibitem{Sc} K. Schacke, "On the kronecker product", Master's thesis, \textit{University of Waterloo} (2004).

\bibitem{Stev} D. Stevanović, and G. Indulal. "The distance spectrum and energy of the compositions of regular graphs." \emph{Applied mathematics letters} 22.7 (2009): 1136-1140.

\bibitem{van}ER. van Dam, JH. Koolen, and H. Tanaka, "Distance-Regular Graphs." \emph{The Electronic Journal of Combinatorics}: EJC (2016): 1-156.

\end{thebibliography}
\end{document}